\documentclass[10pt]{article}
\usepackage{latexsym,amsmath,color,amsthm,amssymb,epsfig,graphicx,mathrsfs}
\usepackage{graphicx}
\usepackage{amssymb}
\usepackage{fullpage}
\usepackage[linktocpage=true]{hyperref}
\usepackage{setspace}
\usepackage{amssymb, amsmath, amsthm, graphicx,mathrsfs}

\def\qed{\hfill\ifhmode\unskip\nobreak\fi\quad\ifmmode\Box\else\hfill$\Box$\fi}
\def\ite#1{\hfill\break${}$\hbox to 50pt {\quad(#1)\hfill}}
\newtheorem{thm}{Theorem}

\newtheorem{lem}[thm]{Lemma}

\newtheorem{claim}[thm]{Claim}

\begin{document}

\title{\vspace{-0.5in} A variation of a theorem by P\'osa}

\author{
{{Zolt\'an F\" uredi}}\thanks{
\footnotesize {Alfr\' ed R\' enyi Institute of Mathematics, Hungary
E-mail:  \texttt{furedi.zoltan@renyi.mta.hu}.
Research supported in part by the Hungarian National Research, Development and Innovation Office NKFIH grant K116769, and  by the Simons Foundation Collaboration Grant 317487.
}}
\and
{{Alexandr Kostochka}}\thanks{
\footnotesize {University of Illinois at Urbana--Champaign, Urbana, IL 61801
 and Sobolev Institute of Mathematics, Novosibirsk 630090, Russia. E-mail: \texttt {kostochk@math.uiuc.edu}.
 Research 
is supported in part by NSF grants  DMS-1266016 and DMS-1600592
and grants 18-01-00353A and 16-01-00499  of the Russian Foundation for Basic Research.
}}
\and{{Ruth Luo}}\thanks{University of Illinois at Urbana--Champaign, Urbana, IL 61801, USA. E-mail: {\tt ruthluo2@illinois.edu}.}}

\date{
         April 12, 2018}

\maketitle

\vspace{-0.3in}

\begin{abstract} A graph $G$ is $\ell$-hamiltonian if for any linear forest $F$ of $G$ with $\ell$ edges, $F$ can be extended to a hamiltonian cycle of $G$.
We give a sharp upper bound for the maximum number of cliques of a fixed size in a non-$\ell$-hamiltonian graph. Furthermore, we prove stability for the bound: if a non-$\ell$-hamiltonian graph contains almost the maximum number of cliques, then it must be a subgraph of one of two examples.

\end{abstract}

\section{Background, $\ell$-hamiltonian graphs}
We use standard notation. In particular, $V(G)$ denotes the vertex set of a graph $G$, $E(G)$ denotes
the edge set of $G$, and $e(G)=|E(G)|$. Also, if $v\in V(G)$, then $N(v)$ denotes the neighborhood of $v$ and
$\deg(v)=|N(v)|$.
Call a graph {\em $\ell$-hamiltonian} if for every linear forest $F$ with $\ell$ edges contained in $G$, $G$
has a hamiltonian cycle containing all edges of $F$. In particular,   `$0$-hamiltonian' means `hamiltonian'.
A well-known sufficient condition for a graph to be
$\ell$-hamiltonian was proved by P\'osa~\cite{Po}:

\begin{thm}[P\'osa~\cite{Po}]\label{Posa2}
Let $n\geq 3$, $1\leq \ell <n$ and let $G$ be an $n$-vertex graph such that
$\deg(u)+\deg(v)\geq n+\ell\;$ for every non-edge $uv$ in $G$.
Then $G$ is $\ell$-hamiltonian.
\end{thm}

A family of extremal non-$\ell$-hamiltonian graphs is as follows. For $n,d,\ell \in \mathbb N$ with $\ell < d \leq \left \lfloor \frac{n+\ell - 1}{2} \right \rfloor$,
let the graph $H_{n,d,\ell}$ be obtained from a copy of $K_{n-d+\ell}$, say with vertex set $A$,
 by adding $d-\ell$ vertices of degree $d$ each of which is adjacent to the same set $B$ of $d$ vertices in $A$ (see the left side of Figure~\ref{f02}).  Let
\begin{equation}\label{equ0}
   h(n,d,\ell):=|E(H_{n,d,\ell)}|={n-d+ \ell \choose 2} + (d-\ell)d.
\end{equation}
For any $\ell < d \leq \lfloor (n+\ell - 1)/2 \rfloor$, graph $H_{n,d,\ell}$  is not $\ell$-hamiltonian: No linear forest $F$ with $\ell$ edges in $G[B]$ can be completed to a hamiltonian cycle of $H_{n,d,\ell}$.
Erd\H{o}s ~\cite{Erdos}  proved the following Tur\' an-type result:
\begin{thm}[Erd\H{o}s~\cite{Erdos}]\label{Erdos} Let $n$  and $d$ be integers with $0< d \leq \left \lfloor \frac{n-1}{2} \right \rfloor$.
If $G$ is a nonhamiltonian graph on $n$ vertices with minimum degree $\delta(G) \geq d$, then
     \[e(G) \leq \max\left\{
   h(n,d,0),h(n, \left \lfloor \frac{n-1}{2} \right\rfloor,0)\right\}.\]
\end{thm}
Graphs  $H_{n,d,0}$ and $H_{n, \left \lfloor \frac{n-1}{2} \right\rfloor,0}$ show the sharpness of Theorem~\ref{Erdos}.

\begin{figure}[!ht]\label{f02}
  \centering

    \includegraphics[width=0.3\textwidth]{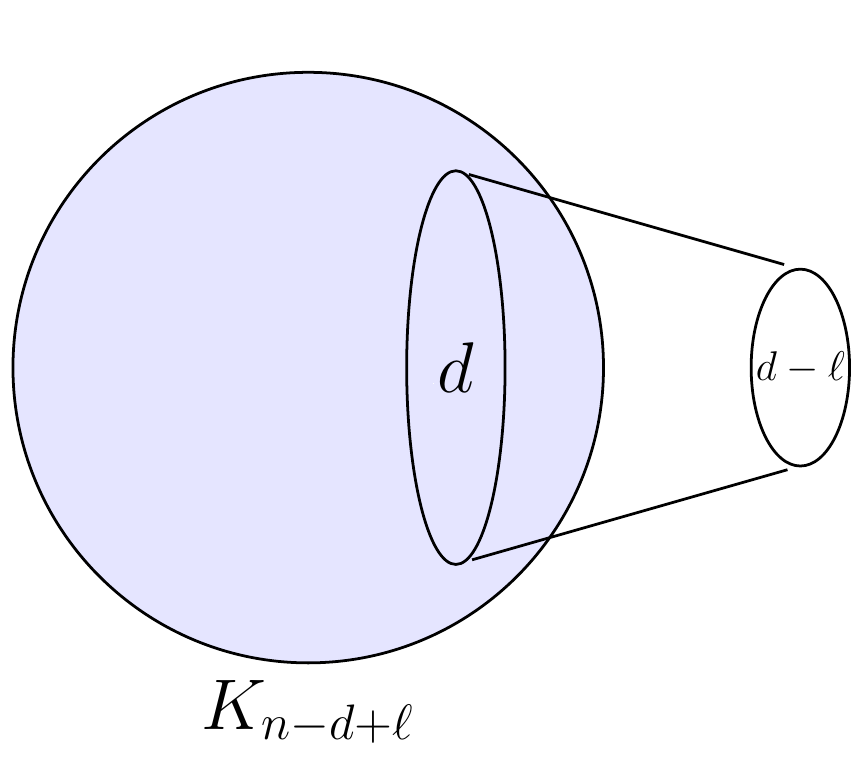} \;\;\;\;\;\;
    \includegraphics[width=0.3\textwidth]{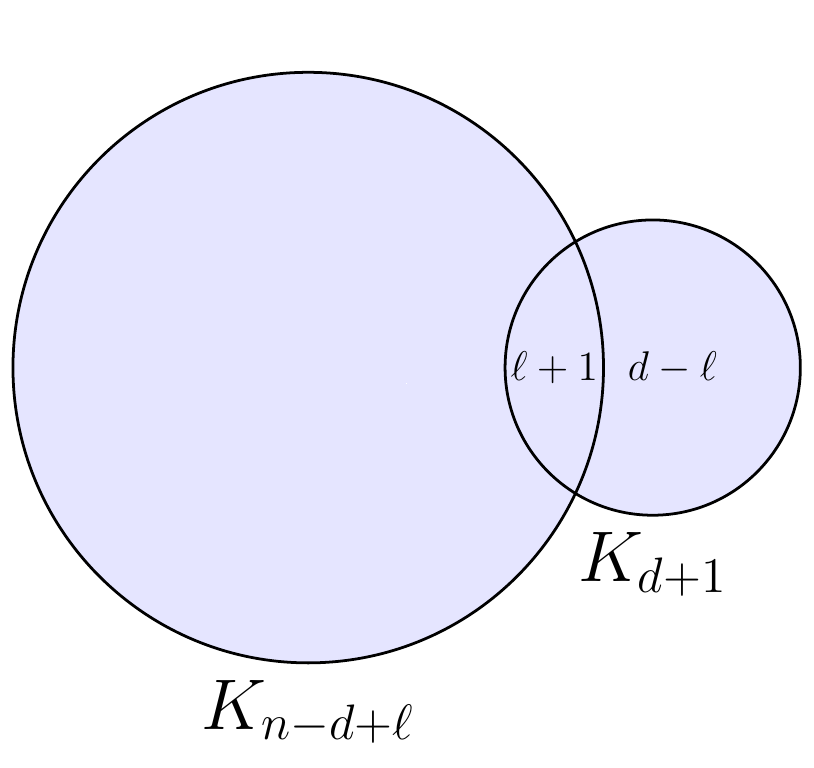}
  \caption{ Graphs $H_{n,d,\ell}$ and $H'_{n,d,\ell}$; shaded ovals denote complete graphs.}
\end{figure}

Another non-$\ell$-hamiltonian graph with  many edges
is  $H'_{n,d,\ell}$ obtained from a copy
 of $K_{n-d+\ell}$ and a copy  of $K_{d+1}$ by identifying $\ell + 1$ vertices  (see Fig.~2, on the right).
Similarly to $H_{n,d,\ell}$, no path of $\ell$ edges spanning the  $\ell + 1$ dominating vertices in $H'_{n,d,\ell}$ can be extended to a hamiltonian cycle.

Li and Ning~\cite{lining} and independently the present authors~\cite{oldmain} proved the following refinement of Theorem~\ref{Erdos}.

\begin{thm}\label{ma}
Let $n\geq 3$ and $d\leq \left \lfloor \frac{n-1}{2} \right \rfloor$.
Suppose that $G$ is an $n$-vertex  nonhamiltonian graph  with minimum degree $\delta(G) \geq d$ such that
\begin{equation}\label{equ2}
    e(G) > \max\left\{ h(n,d+1,0), h(n, \left\lfloor \frac{n-1}{2}\right\rfloor,0)\right\}.
\end{equation}
Then $G$ is a subgraph of either $H_{n,d,0}$ or $H'_{n,d,0}$.
\end{thm}

Recently, Ma and Ning~\cite{maning} extended Theorem~\ref{ma} to graphs with bounded circumference.
Also, Luo~\cite{luo} and Ning and Peng~\cite{ningpeng} bounded the number cliques of given size in graphs with
bounded circumference. The goal of this note is to refine and extend Theorem~\ref{ma} in a different direction---for
non-$\ell$-hamiltonian graphs. One can also view them as an extension of Theorem~\ref{Posa2}.
We state our results in the next section and prove them in the remaining two sections.

\section{Our results}

For a graph $G$, let $N(G, K_r)$ denote the number of copies of $K_r$ in $G$.
In particular, $N(G, K_2) = e(G)$.
Let
$$h_r(n,d,\ell) := N(H_{n, d, \ell}, K_r) = {n - d + \ell \choose r} + (d-\ell){d \choose r-1}.
   $$

We show that classical results easily imply the following extension of Theorem~\ref{Erdos} for non-$\ell$-hamiltonian graphs.
\begin{thm}\label{PP} Let $n, d,r,\ell$ be integers with $r\geq 2$ and $0 \leq \ell < d \leq \left \lfloor \frac{n+\ell-1}{2} \right \rfloor$.
If $G$ is an $n$-vertex graph with minimum degree $\delta(G) \geq d$, and $G$ is not $\ell$-hamiltonian, then
     \[N(G, K_r) \leq \max\left\{ h_r(n,d,\ell),h_r(n, \left \lfloor \frac{n+\ell-1}{2} \right \rfloor,\ell)\right\}.\]
     In particular,
     $$e(G)\leq \max\left\{
   h(n,d,\ell),h(n, \left \lfloor \frac{n+\ell-1}{2} \right\rfloor,\ell)\right\}.$$
The graphs $H_{n,d,\ell}$ and $H_{n, \left \lfloor (n+ \ell - 1)/2 \right \rfloor, \ell}$
  show that this bound is sharp for all
$0\leq \ell < d\leq \left \lfloor \frac{n+ \ell - 1}{2} \right \rfloor$.
\end{thm}

Note that a partial case of Theorem~\ref{PP} (when $\ell=1$ and $r=2$) was proved by Ma and Ning~\cite{maning}.

We also obtain a generalization of Theorem~\ref{ma} which can be viewed as a stability version of Theorem~\ref{PP}.

\begin{thm}\label{ma2}
Let $n\geq 3$, $r\geq 2$  and $\ell < d\leq \left \lfloor \frac{n+\ell-1}{2} \right \rfloor$.
Suppose that $G$ is an $n$-vertex not $\ell$-hamiltonian graph  with minimum degree $\delta(G) \geq d$ such that
\begin{equation}
    N(G, K_r) > \max\left\{ h_r(n,d+1,\ell), h_r(n, \left\lfloor \frac{n+\ell-1}{2}\right\rfloor,\ell)\right\}.
\end{equation}
Then $G$ is a subgraph of either $H_{n,d,\ell}$ or $H'_{n,d,\ell}$.
\end{thm}

\section{Bound on the number of $r$-cliques: Proof of Theorem~\ref{PP} }

Beside the Ore-type condition of  Theorem~\ref{Posa2}, P\'osa~\cite{Po} and independently Kronk~\cite{Kronk} proved the following degree sequence version.

\begin{thm}[P\'osa~\cite{Po}, Kronk~\cite{Kronk}]\label{Posa2k}
Let $G$ be a graph on $n$ vertices and let   $0\leq \ell \leq n-2$.
The following two conditions (together) are sufficient for $G$ to be $\ell$-hamiltonian:\\ {\rm (\ref{Posa2k}.1)} for all integers $k$ with
 $\ell < k < \frac{n+\ell-1}{2} $, the number of points of degree not exceeding $k$  is less than $k-\ell$, \\
{\rm (\ref{Posa2k}.2)} the number of points of degree not exceeding $\frac{n+\ell-1}{2}$  does not exceed $\frac{n-\ell-1}{2}$.
\end{thm}

We need  the following easy claim.
\begin{claim}\label{kcliques}
Let $G$ be an $n$ vertex graph with a set of $s$ vertices with degree at most $t$. Then $N(G, K_r) \leq {n - s \choose r} + s{t \choose r-1}$.
\end{claim}
\begin{proof}
Let $D$ be the set of $s$ vertices with degree at most $t$. Then the number of $K_r$'s disjoint from $D$ is at most   ${n - s \choose r}$. Meanwhile, since each vertex $v$ of $D$ has degree at most $t$, $v$ is contained in at most ${t \choose r-1}$ copies of $K_r$. Summing up over all $v \in D$, we obtain our result.
\end{proof}

The following lemma is a corollary of Theorem~\ref{Posa2k} using Claim~\ref{kcliques}.

\begin{lem}\label{lem8}
 Let $G$ be an $n$-vertex, not $\ell$-hamiltonian graph  with $N(G,K_r) > h_r(n, \left \lfloor \frac{n+ \ell - 1}{2} \right \rfloor, \ell)$.
Then \\
${}$\quad {\rm (\ref{lem8}.1)} \enskip $V(G)$ contains a subset $D$ of $k-\ell$ vertices of degree at most $k$ for some $k$ with $\ell < k \leq  \lfloor \frac{n+ \ell - 1}{2}\rfloor$; \\
${}$\quad {\rm (\ref{lem8}.2)} \enskip $k$ should be less than $\lfloor \frac{n+ \ell - 1}{2}\rfloor$;  \\
${}$\quad {\rm (\ref{lem8}.3)} \enskip $N(G, K_r) \leq h_r(n,k,\ell)$.
   \end{lem}

\begin{proof}
To estimate $N(G,K_r)$ in the cases {\rm (\ref{Posa2k}.2)}  and {\rm (\ref{Posa2k}.1)}  in Theorem~\ref{Posa2k} apply Claim~\ref{kcliques} with the values
  $(s,t)=(\frac{n- \ell +1}{2},\frac{n+ \ell - 1}{2})$ and with $(s,t)=(\frac{n- \ell - 2}{2},\frac{n+ \ell - 2}{2})$, respectively.
In both cases the upper bound for $N(G,K_r)$ is equal to $h(n,\left\lfloor \frac{n+\ell - 1}{2}\right\rfloor ,\ell)$.
Equation~{\rm (\ref{lem8}.3)} is also implied by Claim~\ref{kcliques}.
\end{proof}

For any integer $p\geq 1$
 and real $x$ define ${x\choose p}$ as $x(x-1)\dots (x-p+1)/p!$ for $x\geq p-1$ and $0$ for $x<p-1$.
This function is non-negative and convex  (investigate the second derivative with respect to $x$ when the function is positive).

 For fixed integers $n$, $\ell$, and $r$ with  $0\leq \ell\leq n-1$, and $r\geq 2$, consider the function $h_r(n,x,\ell)$
in the closed interval  $\left[\ell, \left \lfloor \frac{n+\ell-1}{2} \right\rfloor\right]$.
One can show that this function is also convex in $x$ (since both terms are convex), and it is strictly convex where it is positive.
We obtained the following.

\begin{claim}\label{claim10}
Let $J\subseteq \left[\ell, \left \lfloor \frac{n+\ell-1}{2} \right\rfloor\right]$ be a closed interval.
Then  $h_r(n,x,\ell)$ is maximized on $J$ at either of its endpoints.
  \qed
\end{claim}

{\bf Proof of Theorem~\ref{PP}.} Suppose that $G$ is an $n$-vertex, not $\ell$-hamiltonian graph with minimum degree $\delta(G) \geq d$,   for some $1\leq \ell \leq d \leq \left \lfloor \frac{n+\ell - 1}{2} \right \rfloor$.

If $N(G, K_r) \leq h_r(n,\lfloor \frac{n+\ell - 1}{2} \rfloor, \ell)$, then we are done.
Otherwise, {\rm (\ref{lem8}.3)} in Lemma~\ref{lem8}
implies that $ N(G, K_r) \leq h_r(n,k,\ell)$ for some $\ell < k < \lfloor \frac{n+\ell - 1}{2} \rfloor$.
We have that $k\geq \delta(G)\geq d$.
So Claim~\ref{claim10} gives
 \[N(G, K_r) \leq \max_{k\in \left[ d,\lfloor \frac{n+\ell - 1}{2} \rfloor \right]} h_r(n,k,\ell)=\max \left\{ h_r(n,d,\ell), h_r(n, \left\lfloor \frac{n+\ell - 1}{2}\right\rfloor,\ell)\right\}.   \qed
\]

\section{Stability: Proof of  Theorem~\ref{ma2}}

Call an $n$-vertex graph $G$  {\it $\ell$-saturated} if $G$ is not $\ell$-hamiltonian, but each $n$-vertex graph
obtained from $G$ by adding an edge is $\ell$-hamiltonian.

\begin{thm}[Bondy and Chv\'atal (Theorem~9.11 in~\cite{BC76})]\label{D1}
Let $G$ be an  $n$-vertex $\ell$-saturated graph $G$. Then for each non-edge $uv \notin E(G)$, $\deg(u) + \deg(v) \leq n-1+\ell$.
\end{thm}

They observed that the proof  by P\'osa~\cite{Po} yields the following fact:
If $G$ is an  $n$-vertex, not $\ell$-hamiltonian graph, $uv \notin E(G)$, and $\deg(u) + \deg(v) \geq n+\ell$, then
  $G+uv$ is not $\ell$-hamiltonian either.
Since this result implies both Theorems~\ref{Posa2} and~\ref{Posa2k},
  to make this paper self-contained we include here a sketch of their proof.

Suppose on the contrary, that
$G$ has no hamiltonian cycle containing some linear forest on $\ell$ edges $F$ but $G+uv$ has a hamiltonian cycle through $F$.
Then we can order the vertices so that $G$ has a hamiltonian path $w_1w_2\ldots w_n$ where $w_1=u$ and $w_n=v$ containing $F$.
Let $N_G(u)=\{w_{i_1},\ldots,w_{i_k}\}$ where $k= \deg_G(u)$.
If there exists a $1\leq j \leq k$ such that $w_{i_j -1} \in N(v)$ and $w_{i_j}w_{i_j-1} \notin  E(F)$, then
\[w_1\ldots w_{i_j-1} \cup w_{i_j-1} w_n \cup w_n \ldots w_{i_j} \cup w_{i_j}w_1\] is a hamiltonian cycle in $G$ containing $F$.
So either $w_{i_j-1}w_{i_j} \in F$, or $w_{i_j-1} \notin N(v)$ .
Since $F$ contains $\ell$ edges, we have $k-\ell$ choices of $w_{i_j-1}\in V(G)\setminus \{ v\}$ satisfying $vw_{i_j-1} \notin E(G)$.
This gives $\deg_G(v)\leq (n-1)-(k-\ell)$ yielding the contradiction $\deg(u) + \deg(v) \leq n+\ell-1$.

To complete the proof of Theorem~\ref{D1} suppose that $G$ is $\ell$-saturated (note that $G$ is not $\ell$-hamiltonian) and suppose that $uv\notin E(G)$. If  $\deg(u) + \deg(v) \geq n+\ell$, then $G+uv$ is not $\ell$-hamiltonian, a contradiction.    \qed

We show  a useful feature of the structure of saturated  graphs with many edges.

\begin{lem}\label{g-d} Let $G$ be an $\ell$-saturated $n$-vertex  graph with $N(G,K_r) > h_r(n, \left \lfloor \frac{n+ \ell - 1}{2} \right \rfloor, \ell)$.
Then for some $\ell < k < \left\lfloor \frac{n+ \ell - 1}{2} \right\rfloor$, $V(G)$ contains a subset $D$ of $k-\ell$ vertices of degree at most $k$  such that $G - D$ is a complete graph.
\end{lem}
{\bf Proof.}
Apply Lemma~\ref{lem8}  {\rm (\ref{lem8}.1)} to $G$. It says that there is a subset of $k-\ell$ vertices of degree at most $k$  such that $\ell < k \leq \left\lfloor \frac{n+ \ell - 1}{2} \right\rfloor$.
Choose the maximum such $k$, and let $D$ be the set of the vertices with degree at most $k$.
Then Lemma~\ref{lem8}  {\rm (\ref{lem8}.2)} implies that $k <\left\lfloor \frac{n+ \ell - 1}{2} \right\rfloor $.
Then the maximality of $k$ gives that  $|D|=k-\ell$.

Suppose there exist $x, y \in V(G) - D$ such that $xy \notin E(G)$. Among all such pairs, choose $x$ and $y$ with the maximum $\deg(x)$.
Let $D':=V(G)-N(x)-\{x\}$.
Here $|D'|=n-\deg(x)-1>0$.
By Theorem~\ref{D1},
\begin{equation*}
 \mbox{ \em $\deg(z) \leq n  + \ell - 1 - \deg(x) = |D'| + \ell =: k' \;\;$ for all $\;\; z\in D'$.
}
\end{equation*}
So $D'$ is a set of $k'-\ell$ vertices of degree at most $k'$.
Since $y\in D'\setminus D$, $\, k' \geq \deg(y) > k$.
Thus by the maximality of $k$, we get
  $k' = n+\ell - 1 -\deg(x)> \left\lfloor \frac{n+\ell - 1}{2}\right\rfloor$.
Equivalently, $\deg(x) < \lceil \frac{n + \ell - 1}{2}\rceil$, i.e.,  $\deg(x)\leq \left\lfloor \frac{n+\ell - 1}{2}\right\rfloor$.
We also get $|D'|=n-1-\deg(x)>  n-1 -\lceil \frac{n + \ell - 1}{2}\rceil =\lfloor\frac{n+\ell - 1}{2}\rfloor -\ell$.

For all $z \in D'$, either $z \in D$ where $\deg(z)\leq k \leq \left\lfloor\frac{n+\ell - 1}{2}\right\rfloor$, or $z \in V(G) - D$,
and so $\deg(z) \leq \deg(x)\leq \left\lfloor \frac{n+\ell - 1}{2}\right\rfloor$.
It follows that
 $D'$ has a subset of $ \left\lfloor \frac{n+\ell-1}{2}\right\rfloor - \ell$ vertices of degree at most  $\left\lfloor \frac{n+\ell -1}{2}\right\rfloor$.
This contradicts Lemma~\ref{lem8}  {\rm (\ref{lem8}.2)}.

Thus, $G-D$ is a complete graph.\qed

\begin{lem}\label{hnd} Under the conditions of Lemma \ref{g-d}, if $k = \delta(G)$, then $G = H_{n,\delta(G),\ell}$ or $G = H'_{n,\delta(G),\ell}$.
     \end{lem}
{\bf Proof.}
Set $d:=\delta(G)$, and let $D$ be a set of $d-\ell$ vertices with degree at most $d$.
Let $u\in D$. Since $\delta(G)\geq |D|+\ell=d$, $u$  has a neighbor $w$ outside of $D$.
Consider any  $v\in D - \{u\}$.
By Lemma~\ref{g-d}, $w$ is adjacent to all of $V(G) - D - \{w\}$. It also is adjacent  to $u$, therefore its degree is at least $(n-d+\ell - 1) +1 = n-d + \ell$.
We obtain
    \[\deg(w) + \deg(v) \geq (n-d+\ell) + d = n+\ell.\]
Then by~\eqref{D1}, $w$ is adjacent to $v$, and hence $w$ is adjacent to all vertices of $D$.

Let $W$ be the set of vertices in $V(G)-D$ having a neighbor in $D$.
We have obtained that $|W|\geq \ell + 1$ and
\begin{equation}\label{D3}
\mbox{\em
 $N(u)\cap (V(G)-D)=W$ for all $u\in D$.}
\end{equation}
Let $G' = G[D \cup W]$.
Since $|W|\leq \delta(G)$, $|V(G')|\leq 2d-\ell$.
If $|V(G')| = 2d-\ell$, then by~\eqref{D3}, each vertex $u \in D$ has the same $d$ neighbors in $V(G) -D$.
Because $\deg(u) = d$, $D$ is an independent set. Thus $G=H_{n,d,\ell}$. Otherwise,  $(\ell + 1 + d - \ell) \leq |V(G')| \leq 2d - \ell-1$.

If $|V(G')| = d + 1$, then $|W| = \ell + 1$. Because $\delta(G) \geq d$, each vertex in $D$ has at least $d-1$ neighbors in $D$. But this implies that $D$ is a clique, and $G = H'_{n, d,\ell}$.

So we may assume $d+2 \leq |V(G')|\leq 2d-\ell - 1$. That is, $|W| \geq \ell + 2$.
We will show that in this case $G$ is $\ell$-hamiltonian, a contradiction.

Let $F$ be a linear forest in $G$ with $\ell$ edges, set $F_1 := F \cap G'$, $F_2 := F - F_1$.
Let $ab$ be any edge within $G[W]$ such that $ab \notin E(F)$ and $F \cup ab$ is a linear forest in $G$. Such an edge must exist because $G[W]$ is a clique and either $F_1$ is a path that occupies at most $\ell + 1$ vertices in $W$, or $F_1$ is a disjoint union of paths and we can choose $ab$ to join endpoints of two different components of $F_1$.

For any $x,x' \in V(G')$,
    \[\deg_{G'}(x) + \deg_{G'}(x') \geq d + d \geq |V(G')| +\ell + 1.\]
Therefore by Theorem~\ref{Posa2}, $G'$ has a hamiltonian cycle $C$ that passes through $F_1 \cup ab$. In particular, we obtain an $(a,b)$-hamiltonian path $P_1$ in $G'$ which passes through $F_1$.
Since $G'':=G-(V(G')-\{a,b\})$ is a complete graph, it contains an $(a,b)$-hamiltonian path $P_2$ that passes through $F_2$.
Then $P_1 \cup P_2$ is a hamiltonian cycle of $G$ containing $F$, a contradiction. \qed

{\bf Proof of Theorem~\ref{ma2}}.
Let  $G'$ be obtained by adding edges to $G$ until it is $\ell$-saturated.
If $N(G', K_r) \leq h_r(n,\lfloor \frac{n+\ell - 1}{2} \rfloor, \ell)$, then we are done.
Otherwise, by  {\rm (\ref{lem8}.2)} in Lemma~\ref{lem8}
   $G'$ contains a set of $k-\ell$ vertices with degree at most $k$ where $\ell < k < \lfloor (n+\ell-1)/2 \rfloor$.
If $k = d$, then $G' \in \{H_{n,d,\ell},H'_{n,d,\ell}\}$ by Lemma~\ref{hnd}, and thus $G$ is a subgraph of one of these two graphs.
If $k \geq d+1$, then $N(G, K_r) \leq N(G', K_r) \leq h_r(n,k, \ell)$ for some   $d+1\leq  k < \lfloor \frac{n+\ell - 1}{2} \rfloor$ by
 {\rm (\ref{lem8}.3)} in Lemma~\ref{lem8}.

So the convexity by Claim~\ref{claim10} gives that in both cases
 \[N(G, K_r) \leq \max_{k\in \left[ d+1,\lfloor \frac{n+\ell - 1}{2} \rfloor \right]} h_r(n,k,\ell)=\max \left\{ h_r(n,d+1,\ell), h_r(n, \left\lfloor \frac{n+\ell - 1}{2}\right\rfloor,\ell)\right\},
\]
a contradiction. \qed

\bigskip
{\bf Conluding remarks}\smallskip

Since in the proof of Theorem~\ref{PP} for the upper bound for $N(G,K_r)$ we only use the degree sequence of $G$, it seems to be likely that one can obtain similar results for graph classes whose degree sequences are well understood.

Let $P$ be a property defined on all graphs of order $n$ and let $k$ be a nonnegative
integer. Bondy and Chv\'atal~\cite{BC76} call $P$ to be $k$-stable if whenever $G+uv$ has property $P$ and $\deg_G(u)+
\deg_G(v)\geq k$, then $G$ itself has property $P$.
The $k$-closure ${\rm Cl}_k(G)$ of a graph $G$ is the (unique)
smallest graph $H$ of order $n$ such that $E(G)\subseteq E(H)$ and $\deg_H(u)+\deg_H(v) < k$ for all $uv \notin E(H)$.
The $k$-closure can be obtained from $G$ by a recursive procedure of joining nonadjacent vertices with degree-sum at
least $k$. Thus, if $P$ is $k$-stable and ${\rm Cl}_k(G)$ has property $P$, then $G$ itself has property $P$.
They prove $k$-stability (with appropriate values of $k$) for a series of graph properties, e.g.,
  $G$ contains $C_s$ ($k=2n-s$), $G$ contains a path $P_s$ ($k=n-1$), $G$ contains a matching $sK_2$ ($k=2s-1$),
  $G$ contains a spanning $s$-regular subgraph ($k=n+2s-4)$, $G$ is $s$-connected ($k=n+s-2$),
 $G$ is $s$-wise hamiltonian, i.e., every $n-s$ vertices span a $C_{n-s}$ ($k=n+s-2$).

These graph classes are good candidates to find the maximum number of $r$-cliques.
But the proof of stability (like in Theorem~\ref{ma2}) might require more insight.

\end{document}